\documentclass[11 pt]{amsart}
\usepackage{amsmath,amsthm,tikz}
\usepackage{verbatim}
\usepackage{marvosym}
\usepackage{graphicx}
\usepackage{color}
\usepackage{epsfig}
\usepackage{rotating}
\usepackage{lscape}
\usepackage{hyperref}
\usepackage{bm}
\usepackage[all]{xy}

\definecolor{pink}{rgb}{0.858, 0.188, 0.478}

\theoremstyle{plain}
\newtheorem{thm}{Theorem}[section]
\newtheorem{theorem}{Theorem}
\newtheorem{cor}[thm]{Corollary}

\newtheorem{prop}[thm]{Proposition}
\newtheorem{lemma}[thm]{Lemma}

\theoremstyle{definition}

\newtheorem{remark}[thm]{Remark}

 \newcommand{\AAA}{\ensuremath{{\mathbb{A}}}}
 \newcommand{\BB}{\ensuremath{{\mathbb{B}}}}
 \newcommand{\LL}{\ensuremath{{\mathbb{L}}}}

 \newcommand{\FF}{\ensuremath{{\mathbb{F}}}}

\newcommand{\KK}{\ensuremath{{\mathbb K}}}

\newcommand{\QQ}{\ensuremath{\mathbb{Q}}}
\newcommand{\GG}{\ensuremath{\mathbb{G}}}

\newcommand{\ZZ}{\ensuremath{\mathbb{Z}}}

\newcommand{\GL}{\text{GL}}

\makeatletter
\newtheorem*{rep@theorem}{\rep@title}
\newcommand{\newreptheorem}[2]{%
\newenvironment{rep#1}[1]{%
 \def\rep@title{#2 \ref{##1}}%
 \begin{rep@theorem}}%
 {\end{rep@theorem}}}
\makeatother

\newreptheorem{theorem}{Theorem}
\newreptheorem{lemma}{Lemma}
\newreptheorem{corollary}{Corollary}

\newcounter{nootje}
\begin{document}
\title[Lower bounds for heights in relative Galois extensions]{Lower bounds for heights in relative Galois extensions}
\author{Shabnam Akhtari}
\address{Department of Mathematics, University of Oregon, Eugene, Oregon 97402 USA}
\email{akhtari@uoregon.edu}
\author{Kevser Akta\c{s}}
\address{Department of Mathematics Education, Gazi University, Ankara, 06500, Turkey}
\email{kevseraktas@gazi.edu.tr}
\author{Kirsti Biggs}
\address{School of Mathematics, University of Bristol, University Walk, Clifton, Bristol, BS8 1TW, United Kingdom}
\email{kirsti.biggs@bristol.ac.uk}
\author{Alia Hamieh}
\address{Department of Mathematics and Computer Science, University of Lethbridge, Lethbridge, Alberta, T1K3M4 Canada}
\email{alia.hamieh@uleth.ca}
\author{Kathleen Petersen}
\address{Department of Mathematics, Florida State University, Tallahassee, Florida 32306 USA}
\email{petersen@math.fsu.edu}
\author{Lola Thompson}
\address{Department of Mathematics, Oberlin College, Oberlin, OH 44074 USA}
\email{lola.thompson@oberlin.edu}

\begin{abstract}

The goal of this paper is to obtain lower bounds on the height of an algebraic number in a relative setting, extending previous work of Amoroso and Masser.  Specifically, in our first theorem we obtain an effective bound for the height of an algebraic number $\alpha$ when the base field $\mathbb{K}$ is a number field and $\mathbb{K}(\alpha)/\mathbb{K}$ is Galois.  Our second result establishes an explicit height bound for any non-zero element  $\alpha$ which is not a root of unity in a Galois extension $\FF/\KK$, depending on the degree of $\KK/\QQ$ and the number of conjugates of $\alpha$ which are multiplicatively independent over $\KK$. As a consequence, we obtain a height bound for such $\alpha$ that is independent of the multiplicative independence condition.

\end{abstract}


\maketitle

\section{Introduction}

Consider the non-constant polynomial
\[ P(x)=c_dx^d+c_{d-1}x^{d-1}+\dots + c_1x+c_0=  c\prod_{i=1}^d (x-r_i).\]
The Mahler measure of $P(x)$ is defined as
\[ M(P) = \exp\Big(  \int_0^1 \log | P(e^{2\pi i t})| \; dt \Big),\]
the geometric mean of $|P(z)|$ for $z$ on the unit circle.
By Jensen's formula, this is equivalent to 
\[  M(P) = |c| \prod_{|r_i|\geq 1} |r_i|.\]
If $P(x)$ has integer coefficients, then $M(P)\geq 1$;  by a result of Kronecker, $M(P)=1$ 
exactly when $P(x)$ is a power of $x$ times a product of cyclotomic polynomials.

Given an algebraic number $\alpha$, we let $d= [\QQ(\alpha):\QQ]$ be its degree over $\QQ$. We will use $M(\alpha)$ to denote the Mahler measure of the minimal polynomial of $\alpha$ over $\ZZ$. We will formulate our results in terms of the  Weil height of $\alpha$, defined to be
\[h(\alpha) =\tfrac1d \log M(\alpha).\]

In 1933 Lehmer asked whether there are monic integer polynomials whose Mahler measure is arbitrarily close to 1. For the polynomial $L(x) = x^{10} +x^9 -x^7 -x^6 -x^5 -x^4 -x^3 +x+1$ (now called Lehmer's polynomial), he calculated 
$ M(L) = 1.176280818 \dots,$
which is still the smallest value of $M(P) > 1$ known for $P\in \ZZ[x]$. Although he did not make a conjecture, the statement that there exists a constant $\delta>0$ such that the Mahler measure of any polynomial in $\ZZ[x]$ is either 1 or is greater than $1+\delta$ has become known as Lehmer's conjecture. In terms of height, Lehmer's conjecture states that there is a universal constant $c_0>0$ such that if $\alpha$ is a non-zero  algebraic number which is not a root of unity then 
\[ h(\alpha) \geq \frac{c_0}{d}.\]

In 1971 Blanksby and Montgomery \cite{MR0296021} and later Stewart \cite{MR507748} produced bounds for the Mahler measure of such algebraic numbers.  These bounds inspired the work of Dobrowolski \cite{MR543210} who, in 1979, proved for $d\geq 2$ that 
\[ M(\alpha) > 1+ \frac1{1200} \Big( \frac{\log \log d}{\log d}\Big)^3.\]
 Many of the best bounds are modifications of Dobrowolski's bound.  The constants in these bounds have been improved over the years, but the dependence on the degree (for general polynomials) has remained.
Of note, in 1996 Voutier \cite{MR1367580} used elementary techniques to show that for $d\geq2$, we have
\begin{equation}\label{Vouineq}
 h(\alpha) >\frac1{4d} \Big( \frac{\log \log d}{\log d}   \Big)^3.
 \end{equation}
 (Dobrowolski's bound, when translated into a statement about Weil height, has a similar form.) Voutier also showed that for $d\geq2$, we have \begin{equation}\label{Vouineq1}h(\alpha) >  \frac{2}{d\left(\log 3d\right)^3}, \end{equation}
 which gives a better lower bound than \eqref{Vouineq} for small values of $d$. For more details on the history of Lehmer's conjecture and related problems, see the excellent survey paper of Smyth \cite{MR2428530}.

Lehmer's conjecture has been proven in certain settings. Notably, Breusch \cite{MR0045246} and Smyth \cite{MR0289451} independently proved it for non-reciprocal polynomials. More recently,  Borwein, Dobrowolski and Mossinghoff \cite{MR2373144} proved it for many infinite families of polynomials, including polynomials with no cyclotomic factors and all odd coefficients. (Their result therefore proves Lehmer's conjecture for the Littlewood polynomials, namely those polynomials whose coefficients are $\pm 1$.) 

Results also exist concerning height bounds for $\alpha$ with certain properties. For example, Amoroso and David \cite{MR1713323} have proven that there is an absolute constant $c$ such that if $\QQ(\alpha)/\QQ$ is Galois, and $\alpha$ is not a root of unity, then $h(\alpha)\geq cd^{-1}$. This proves Lehmer's conjecture for such $\alpha$. Moreover, if $\alpha$ is any non-zero algebraic number that lies in an abelian extension of $\QQ$, then Amoroso and Dvornicich \cite{MR1740514} have shown that the height of $\alpha$ is greater than the constant $(\log 5)/12$.

Amoroso and Masser \cite{AmMa} improved upon the bounds in \cite{MR1713323}   for the case where $\QQ(\alpha)/\QQ$ is Galois. They  showed that, for any $\epsilon$,  the height of $\alpha$ is bounded below by $c(\epsilon)d^{-\epsilon}$.  Our first theorem is a generalization of this result to the case when $\alpha$ generates a Galois extension of an arbitrary number field.
\begin{theorem}\label{Thm1}
Let $\epsilon>0$ be given. Let $\alpha$ be a non-zero algebraic number, not a root of unity, such that $[\QQ(\alpha):\QQ]\geq2$ and $\KK(\alpha)/\KK$ is Galois for some number field $\KK$.  Let $\delta$ be the degree of $\alpha$ over $\KK$. 
Then there is an effectively computable constant $c(\epsilon,\KK)>0$ such that 
\[ h(\alpha)\geq c(\epsilon,\KK) \delta^{-2\epsilon}.\]
\end{theorem} 

Relative height bounds for $\alpha$ in a number field $\KK$ which is abelian over $\LL$ are given in \cite{MR1817715} and \cite{MR2349643}.  These bounds are similar in shape to Dobrowolski's bound.

Theorem~\ref{Thm1} determines bounds for $h(\alpha)$ when $\KK(\alpha)/\KK$ is Galois, and therefore when $\QQ(\alpha)/\KK$ is Galois.  Our next theorem  determines height bounds for any element $\alpha$ in a Galois extension $\FF$ of $\KK$ which is non-zero and not a root of unity.  This is a generalization of Theorem 3.1 in \cite{AmMa}.
\begin{theorem}\label{Thm2}
Let $\KK$ be a number field with degree $\tau$ over $\QQ$. For any positive integer $r\geq1$ and any $\epsilon>0$ there is a positive effective constant $c(\epsilon,r,\tau)$ with the following property. Let $\FF/\KK$ be a Galois extension of relative degree $\eta$, and suppose $\alpha\in\FF^{*}$ is not a root of unity. Assume that $r$ conjugates of $\alpha$ over $\KK$ are multiplicatively independent. Then \[h(\alpha)\geq c(\epsilon,r,\tau)\eta^{-\frac{1}{r+1}-\epsilon}.\]
\end{theorem}
Theorem~\ref{Thm2} is proven in Section~\ref{ExplicitSection}, where the explicit constants are presented. Taking $r=1$, we have the following corollary as an immediate consequence.

\begin{cor}\label{corollary:Thm2}
{\em
For any $\epsilon>0$ there is a positive effective constant $c(\epsilon,\tau)$ with the following property. Let $\FF/\KK$ be a Galois extension, with $ [\FF: \KK] = \eta$, and suppose $\alpha\in\FF^{*}$ is not a root of unity. Then \[h(\alpha)\geq c(\epsilon,\tau)\eta^{-\frac{1}{2}-\epsilon}.\]}
\end{cor}

The present paper closely follows and builds on the work of Amoroso and Masser in \cite{AmMa}.

\section{Preliminaries}

In this section, we collect results that will be used in the proofs of Theorem~\ref{Thm1} and Theorem~\ref{Thm2}. 

\subsection{Finite linear groups}

We will require a bound on the size of finite subgroups of $\GL_n(\ZZ)$ in the proof of Lemma \ref{lemma:2.1}.   We now establish this bound, following the work of Serre \cite{Ser}.

\begin{prop}[Serre]\label{thm:serre}
Let $A$ be an abelian variety, and let $u$ be an automorphism of $A$ of finite order. Let $n\geq2$ be a positive integer such that $u\equiv 1\mod n$. If $n=2$, then $u^2=1$. Otherwise, we have $u=1$. 
\end{prop}

The proof of Lemma \ref{lemma:2.1} will use the following well-known corollary to Proposition \ref{thm:serre} (see also \cite[Remark 2.3]{AmMa}).
\begin{cor}\label{cor:serre}
Let H be a finite subgroup of $\mathrm{GL}_{\rho}(\ZZ)$. The reduction modulo 3 homomorphism $\phi_{3}: H\rightarrow\mathrm{GL}_{\rho}\left(\ZZ/3\ZZ\right)$ is injective.  As a result, the order of a finite subgroup of $\mathrm{GL}_{\rho}(\ZZ)$ is less than $3^{\rho^2}$.
\end{cor}

\begin{proof}
Let $u$ be an element in $\ker(\phi_{3})\subset H$. Then $u$ has finite order and $u\equiv I_{\rho}  \mod 3$, where $I_{\rho}$ is the $\rho\times \rho$ identity matrix. By Proposition \ref{thm:serre}, we have $u=I_{\rho}$. This establishes that $\phi_3$ is injective.  We conclude that the order of $H$ is at most $|\GL_{\rho}(\ZZ/3\ZZ)|$, which is less than $3^{\rho^2}$. 
\end{proof} 

\begin{remark}
In an unpublished paper from 1995, Feit \cite{Fei} shows that the maximal order of a finite subgroup of $\mathrm{GL}_{\rho}(\QQ)$ is $2^{\rho} \rho!$, except when $\rho=2,4,6,7,8,9,10$.  He further shows that for these exceptional cases, the maximal order is 
\[ 12, 1152, 103680, 2903040, 696729600, 1393459200, 8360755200,\] respectively.  Therefore, the maximal order of a finite subgroup is at most $\tfrac{135}2 2^\rho \rho!$ for all $\rho$.  See \cite{MR2082091} for more information about these subgroups. Additionally, in 1997, Friedland showed in \cite{MR1443385} that the orthogonal groups are the maximal subgroups for $\rho$ large enough.
\end{remark}

\
\subsection{Height of algebraic numbers}

We will use the following auxiliary height bounds in our proofs of Theorem~\ref{Thm1} and Theorem~\ref{Thm2}.

The first  is Corollary 1.6 of \cite{MR2914852}.
\begin{prop}[Amoroso-Viada]\label{AmVi1.6}
Let $\alpha_1,\dots,\alpha_n$ be multiplicatively independent algebraic numbers in a number field $\AAA$ of degree $D=[\AAA:\QQ]$. Then
\[ h(\alpha_1)\dots h(\alpha_n)\geq D^{-1}(1050n^5 \log 3D)^{-n^2 (n+1)^2}.\]
\end{prop}

 The following result is Th\'eor\`eme 1.3 from \cite{MR2349643}.
\begin{prop}[Amoroso-Delsinne]\label{AmDel1.3}
Let $\alpha$ be a non-zero algebraic number which is not a root of unity. For every abelian extension $\AAA$ of $\BB$, we have
\[
h(\alpha) \geq \frac{(g(\tau)\Delta)^{-c}}{D} \frac{\left(\log \log 5D\right)^3}{\left(\log 2D\right)^4},
\]
where $c$ is an absolute, strictly positive constant, $\Delta$ is the absolute value of the discriminant of $\BB$ over $\QQ$, $\tau=[\BB:\QQ]$, $D=[\AAA(\alpha):\AAA]$, and $g(\tau)=1$ if there exists a tower of extensions 
\[ \QQ=\BB_0\subset \BB_1\subset \dots \subset \BB_m=\BB\]
with $\BB_i/\BB_{i-1}$ Galois for $i=1,\dots m$, and $g(\tau)=\tau!$ otherwise.
\end{prop}

\begin{remark}
The constant $c$ in Proposition \ref{AmDel1.3} depends on a number of constants defined in \cite{MR2349643}, as well as on constants from papers of Friedlander \cite{MR566874} and Stark \cite{MR0342472}. 

\end{remark}

Finally, we will use Th\'eor\`eme 1.6 of \cite{MR2567747}, in which $\QQ^{\rm{ab}}$ denotes the maximal abelian extension of $\QQ$, and $\GG_m(\overline{\QQ})$ denotes the multiplicative group of  $\overline{\QQ}$.
\begin{prop}[Delsinne]\label{Del1.6}
For any positive integer $n$, there exists an effectively computable constant $c(n)>0$ depending only on $n$ for which the following property holds. Let $\bm{\alpha}=(\alpha_1,\dots,\alpha_n)\in\GG_m^n(\overline{\QQ})$. If
\[\prod_{i=1}^n h(\alpha_i)\leq \big(c(n)[\QQ^{\rm{ab}}(\bm{\alpha}):\QQ^{\rm{ab}}](\log(3[\QQ^{\rm{ab}}(\bm{\alpha}):\QQ^{\rm{ab}}]))^{\kappa(n)}\big)^{-1},\]
where $\kappa(n)=3n(2(n+1)^2(n+1)!)^n$, then $\bm{\alpha}$ is contained in a torsion subvariety $B$ for which
\[(\deg B)^{1/\rm{codim}(B)}\leq c(n)[\QQ^{\rm{ab}}(\bm{\alpha}):\QQ^{\rm{ab}}]^{\eta(n)}(\log(3[\QQ^{\rm{ab}}(\bm{\alpha}):\QQ^{\rm{ab}}]))^{\mu(n)},\]
where \[\eta(n)=(n-1)!\bigg(\sum_{i=0}^{n-3}\frac{1}{i!}+1\bigg)+n-1\] and $\mu(n)=8m!(2(n+1)^2(n+1)!)^n$.

In fact, we may take $c(n)=(2n^{2})^{n}\exp\left(64n^{2}n!\left(2(n+1)^{2}(n+1)!\right)^{2n}\right)$.
\end{prop}
Notice that if $\alpha_{1},\dots,\alpha_{n}$ are multiplicatively independent, then $\bm{\alpha}=(\alpha_1,\dots,\alpha_n)$ cannot be contained in a torsion subvariety. This simple observation yields the following corollary to  Proposition~\ref{Del1.6}.
\begin{cor}\label{cor:Del1.6}
Let $n$ be a positive integer, and let $\alpha_{1},\dots,\alpha_{n}$ be multiplicatively independent algebraic numbers. Then there exists an effectively computable constant $c(n)>0$ depending only on $n$ for which \[\prod_{i=1}^n h(\alpha_i) > \big(c(n)[\QQ^{\rm{ab}}(\bm{\alpha}):\QQ^{\rm{ab}}](\log(3[\QQ^{\rm{ab}}(\bm{\alpha}):\QQ^{\rm{ab}}]))^{\kappa(n)}\big)^{-1},\]
where $\kappa(n)=3n(2(n+1)^2(n+1)!)^n$.
\end{cor}

\subsection{Estimates for $\phi(n)/n$}

We will make use of the following lower bound for Euler's totient function, which is a slightly weaker version of \cite[Theorem 15]{MR0137689}. 
\begin{prop}\label{prop:totientlb} For all natural numbers $n\geq3$, we have
\[
\frac{\phi(n)}{n} > \frac{1}{\exp\left(\gamma\right) \log \log n + \frac{3}{\log \log n} },
\]
where $\gamma$ is Euler's constant.
\end{prop}

The following lower bound for $\phi(n)^{1+\epsilon}/n$ will be useful in making the lower bound constants explicit in the proofs of both of our main theorems.

\begin{lemma}\label{lem:phi}
For any $\epsilon>0$, there is an effective constant $C(\epsilon)$  such that  \[\frac{\phi(n)^{1+\epsilon}}{n}\geq C(\epsilon)\]
for all $n\geq 3.$ Specifically, one can take 
\[ C(\epsilon) = \left( \frac{(\log \log 3){\left(\exp(2)\frac{\epsilon}{2+2\epsilon}\right)^{\sqrt{\frac{\epsilon}{2+2\epsilon}}+1}}}{{\exp\left(\gamma\right)}+3^{\frac{1}{1+\epsilon}} {\left(\exp(2)\frac{\epsilon}{2+2\epsilon}\right)^{\sqrt{\frac{\epsilon}{2+2\epsilon}}+1}} }\right)^{1+\epsilon}.\]
\end{lemma}
\begin{proof}
By Proposition \ref{prop:totientlb}, for all  $n\geq3$ we have 
\[
\frac{\phi(n)}{n} > \frac{\log \log n}{\exp\left(\gamma\right)\left(\log \log n\right)^{2}+3}.
\]
We use the fact that $\log x \leq \frac{x^{\theta}}{\exp(1)\theta}$ for any $\theta>0$ to replace the power of $\log \log n$ in the denominator and conclude that  
\[
\frac{\phi(n)}{n} > \frac{\log \log 3}{\exp\left(\gamma\right)\frac{n^{2\theta^2}}{\left(\exp(1)\theta\right)^{2\theta+2}}+3}.\]
Hence, 
\[
\frac{\phi(n)}{n^{1-2\theta^2}}\geq \frac{\log \log 3}{\frac{\exp\left(\gamma\right)}{\left(\exp(1)\theta\right)^{2\theta+2}}+\frac{3}{n^{2\theta^2}}}\geq \frac{\log \log 3}{\frac{\exp\left(\gamma\right)}{\left(\exp(1)\theta\right)^{2\theta+2}}+\frac{3}{3^{2\theta^2}}},
\] 
which implies that 
\[
\frac{\phi(n)^{1+\frac{2\theta^2}{1-2\theta^2}}}{n}\geq\left( \frac{\log \log 3}{\frac{\exp\left(\gamma\right)}{\left(\exp(1)\theta\right)^{2\theta+2}}+3^{1-2\theta^2}}\right)^{1+\frac{2\theta^2}{1-2\theta^2}}.
\] 
Choosing $\theta$ such that $2\theta^2=\frac{\epsilon}{1+\epsilon}$ completes the proof.
\end{proof}

\remark{Our lemma holds for all $n \geq 3.$ By Mertens' theorem (see, for example, \cite[Theorem 3.15]{Pol}), $\phi(n)/n \sim 1/(\exp(\gamma) \log \log n)$ as $n \rightarrow \infty$. Using this, one can obtain sharper lower bounds for $n$ ``sufficiently large.''}

\section{Some Useful Lemmas}\label{relativeGal}

In this section we prove two lemmas that will be useful in the proof of Theorem~\ref{Thm1}.

\begin{lemma}\label{lemma:2.1}
Let $\FF/\KK$ be a Galois extension. Assume that $\alpha\in \FF^*$ is not a root of unity, let $\alpha_1, \dots, \alpha_{\delta}$ be the conjugates of $\alpha$ over $\KK$, and let $\rho$ be the multiplicative rank of this set of conjugates. Let $e$ be the order of the group of roots of unity in $\FF$, so that $\QQ(\zeta_e)\subset \FF$. 
Then there exists a subfield $\LL$ of $\FF$ which is Galois over $\KK$ of relative degree $[\LL:\KK]= n \leq n(\rho)< 3^{\rho^2}$, and $\alpha^e\in \LL$. 
\end{lemma}

\begin{proof}
Let $\beta_i=\alpha_i^e$, and $\LL=\KK(\beta_1, \dots, \beta_{\delta})\subseteq \FF$. Then, by construction, $\LL$ is Galois over $\KK$ and $\alpha^e\in \LL$. Consider 
the multiplicative group \[ \mathcal{M}=\{ \beta_1^{a_1}\cdot \dots \cdot \beta_{\delta}^{a_{\delta}}:a_i\in \ZZ\},\] which  is a $\ZZ$-module that is multiplicatively spanned by $\{\beta_{1},\beta_{2},\dots,\beta_{\delta}\}$.  First, we will show that $\mathcal{M}$ is a free $\ZZ$-module of rank $\rho$. It is enough to show that $\mathcal{M}$ is torsion-free as the fact that $\rho$ is the multiplicative rank of $\{\alpha_1,\dots, \alpha_{\delta}\}$ implies that it is also the multiplicative rank of $\{\beta_1,\dots, \beta_{\delta}\}$. Assume for the sake of contradiction that there exists an $x\in \mathcal{M}$ such that $x\neq 1$ and $x^{n}=1$ for some positive integer $n>1$. Then $x=\beta_1^{a_1}\cdot \dots \cdot \beta_{\delta}^{a_\delta}$ for some $a_{1},\dots,a_{\delta}\in\ZZ$. Since $x^{n}=1$, we get 
\[ \left(\beta_1^{a_1}\cdot \dots \cdot \beta_{\delta}^{a_\delta}\right)^{n}=\left(\alpha_1^{a_1}\cdot \dots \cdot \alpha_{\delta}^{a_{\delta}}\right)^{ne}=1.\]
 Hence, $y=\alpha_1^{a_1}\cdot \dots \cdot \alpha_{\delta}^{a_{\delta}}$ is a root of unity. Since $y\in \FF$ and $e$ is the order of the group of roots of unity in $\FF$, it follows that $x=y^{e}=1$, contrary to our assumption.  

Since $\mathrm{Gal}(\LL/\KK)$ acts on $\mathcal{M}$ by permuting the $\alpha_i$, this action defines an injective homomorphism from $\mathrm{Gal}(\LL/\KK)$ to $\mathrm{GL}_{\rho}(\ZZ)$. This implies that the finite group $\mathrm{Gal}(\LL/\KK)$ is isomorphic to a finite subgroup of $\mathrm{GL}_{\rho}(\ZZ)$. 
By Corollary~\ref{cor:serre} the order of a finite subgroup of $\mathrm{GL}_{\rho}(\ZZ)$ is bounded by $n(\rho)$ which   is at most $3^{\rho^{2}}$. We conclude that $[\LL:\KK]\leq n(\rho)<3^{\rho^{2}}$.
\end{proof}

\begin{lemma}\label{lemma:orderlemma}
Let $\epsilon>0$ be given.  Let $\KK$ be a number field.  Assume that $\alpha$ is a non-zero algebraic number, not a root of unity, such that  $\KK(\alpha)/\KK$ is Galois.  Let $\delta$ be the degree of $\alpha$ over $\KK$. 
Further, let  $e$ be the order of the group of roots of unity in $\KK(\alpha)$, $f$ be the order of the group of roots of unity in $\KK$, $\tau=[\KK:\QQ]$, and let  $\rho$ be the multiplicative rank of the conjugates of $\alpha$ over $\KK$. 
Then
\[ 
[\KK(\alpha):\KK(\zeta_e)] \leq   \delta^{\epsilon}  C_4(\KK,\epsilon),
\]
with $C_4(\KK,\epsilon) = \frac{1}{C(\epsilon)}   n(\rho) f \tau^{1+\epsilon}$. We have $C(\epsilon)$ as in Lemma~\ref{lem:phi} unless $e/f\in\{1,2\}$ in which case we take $C(\epsilon)=1/2$.

\end{lemma}

\begin{proof}
We begin by obtaining a few inequalities, proving \eqref{eq:degreephi} and \eqref{eq:phibounds}  below. 
By the second isomorphism theorem, we have 
\begin{align*} 
 [\KK(\zeta_e):\KK] & =[\QQ(\zeta_e):\KK\cap \QQ(\zeta_e)] =
 \frac{[\QQ(\zeta_e):\QQ(\zeta_f)]}{[\KK\cap \QQ(\zeta_e):\QQ(\zeta_f)]} \\
& =\frac{\phi(e)/\phi(f)}{[\KK\cap \QQ(\zeta_e):\QQ(\zeta_f)]}.
\end{align*}
Since $[\KK\cap \QQ(\zeta_e):\QQ(\zeta_f)]\leq [\KK:\QQ(\zeta_f)]$, we conclude that
\[ 
 [\KK(\zeta_e):\KK]  \geq \frac{\phi(e)/\phi(f)}{[\KK:\QQ(\zeta_f)]}.
\]
It follows from the fact that $\phi$ is multiplicative that
\begin{equation}\label{eq:degreephi}
 [\KK(\zeta_e):\KK]\geq \frac{\phi(e)/\phi(f)}{[\KK:\QQ(\zeta_f)]}  \geq \frac{ \phi(\tfrac{e}{f})}{[\KK:\QQ(\zeta_f)]}.  
 \end{equation}

Next, we will show   
\begin{equation}\label{eq:phibounds}
\frac{\tfrac{e}{f}}{\phi(\tfrac{e}{f})}\leq \frac{[\KK:\QQ(\zeta_{f})]^{\epsilon}}{C(\epsilon)}\delta^{\epsilon}.
\end{equation}
By Lemma \ref{lem:phi}, if $e/f\geq 3$ we have the upper bound 
\[ \frac{\tfrac{e}{f}}{\phi (\tfrac{e}{f})}\leq \frac{1}{C(\epsilon)}\left(\phi(\tfrac{e}{f})\right)^{\epsilon}.\] 
If $e/f \in\{1,2\}$, we can take $C(\epsilon) = 1/2$ and this is still satisfied. 
Again appealing to the multiplicativity of $\phi$, we have 
\[\phi(\tfrac{e}{f})\leq \frac{\phi(e)}{\phi(f)}=[\QQ(\zeta_{e}):\QQ(\zeta_{f})]\leq [\KK(\zeta_{e}):\KK][\KK:\QQ(\zeta_{f})]\leq \delta[\KK:\QQ(\zeta_{f})].\]
Hence, \eqref{eq:phibounds} follows by combining these two inequalities.

Now we will proceed to prove the bound for $[\KK(\alpha):\KK(\zeta_e)]$ stated in the lemma.
By Lemma~\ref{lemma:2.1} there is a subfield $\LL$ of $\KK(\alpha)$ which is Galois over $\KK$, contains $\alpha^e$ and \begin{align}\label{eq:degreen(rho)} [\LL:\KK]=n\leq n(\rho).\end{align} Thus, we have $\KK \subseteq \LL \subseteq \KK(\alpha),$ so $\KK(\alpha) = \LL(\alpha).$

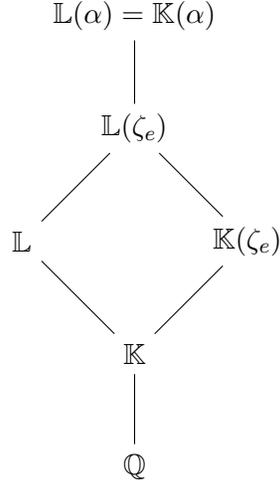
\[\begin{tikzpicture}[node distance = 1.5cm, auto]
      \node (Q) {$\QQ$};
      \node (K) [above of=Q] {$\KK$};
      \node (L) [above of=K, left of=K] {$\LL$};
      \node (Kzeta) [above of=K, right of=K] {$\KK(\zeta_{e})$};
      \node (Lzeta) [above of=K, node distance = 3cm] {$\LL(\zeta_{e})$};
      \node (Lalpha) [above of=Lzeta] {$\LL(\alpha)=\KK(\alpha)$};
      \draw[-] (Q) to node {} (K);
      \draw[-] (K) to node {} (L);
      \draw[-] (K) to node {} (Kzeta);
      \draw[-] (Kzeta) to node {} (Lzeta);
      \draw[-] (L) to node {} (Lzeta);
      \draw[-] (Lzeta) to node {} (Lalpha);
      \end{tikzpicture}\]

Let $e'=[\LL(\alpha):\LL]$. Since the minimal polynomial for $\alpha$ over $\LL$ divides  $x^e-\alpha^e$, we conclude that $e'\leq e$.
Using multiple applications of the tower law, we have
\begin{align*} 
[\KK(\alpha):\KK(\zeta_e)] & =[\LL(\alpha):\LL(\zeta_e)][\LL(\zeta_e):\KK(\zeta_e)]\\
& = [\LL(\zeta_e):\KK(\zeta_e)]\frac{[\LL(\alpha):\LL]}{[\LL(\zeta_e):\LL]} \\
& =e'  \frac{[\LL(\zeta_e):\KK(\zeta_e)]}{[\LL(\zeta_e):\LL]} =
e' \frac{[\LL:\KK]}{[\KK(\zeta_e):\KK]}. 
\end{align*}
By \eqref{eq:degreen(rho)}, we see that
\[
[\KK(\alpha):\KK(\zeta_e)] \leq e' \frac{n(\rho)}{[\KK(\zeta_e):\KK]} .
\]
Using   \eqref{eq:degreephi} we have 
\[ [\KK(\alpha):\KK(\zeta_e)] \leq  e' \frac{n(\rho)}{[\KK(\zeta_e):\KK]} \leq \frac{e'}{  \phi(\tfrac{e}{f})} n(\rho)[\KK:\QQ(\zeta_f)].\]
Since $e'\leq e$, we conclude that $e' \leq \frac{e}{f} f$, and hence
\[
[\KK(\alpha):\KK(\zeta_e)] \leq   \frac{ \tfrac{e}{f}}{  \phi(\tfrac{e}{f})}   n(\rho) f [\KK:\QQ(\zeta_f)].
\]
Combining this  bound with \eqref{eq:phibounds} shows that
\[ 
[\KK(\alpha):\KK(\zeta_e)] \leq   \frac{1}{C(\epsilon)}\delta^{\epsilon}  n(\rho) f [\KK:\QQ(\zeta_{f})]^{1+\epsilon}\leq   \delta^{\epsilon}  C_4(\KK,\epsilon)
\]
with $C_4(\KK,\epsilon) = \frac{1}{C(\epsilon)}   n(\rho) f \tau^{1+\epsilon}$, as needed.

\end{proof}

\section{Proof Of Theorem 1}

In this section, we present the proof of Theorem \ref{Thm1}, which generalizes Theorem 3.3 of \cite{AmMa}.

\begin{proof}
Let $\epsilon>0$ be given,  let $r$ be the smallest integer greater than $1/\epsilon$, and let $\tau=[\KK:\QQ]$. 
First consider the case when $r>\delta$, so that  
\[ \delta\leq  r-1\leq \tfrac{1}{\epsilon}<r.\] 
We will show that $h(\alpha) \gg_{\tau,\epsilon}1$. For $d\geq2$, using equation \eqref{Vouineq1}, we obtain
\[ h(\alpha) \geq \frac{2}{d (\log 3d)^3}. \]
The function $f_1(x)=\frac{2}{x(\log 3x)^3}$ is decreasing for $x\geq 1$.  
Since $\QQ(\alpha)\subset \KK(\alpha)$ we have  
\[ d=[\QQ(\alpha):\QQ]\leq [\KK(\alpha):\QQ]=[\KK(\alpha):\KK][\KK:\QQ]=\delta \tau\leq \tau/\epsilon.\]
Therefore,
\[ h(\alpha) \geq f_1(\tau/\epsilon)=  \frac{2}{\tau/\epsilon (\log 3 \tau/\epsilon)^3}.\]

We can often improve upon this lower bound. Using equation \eqref{Vouineq} for $d\geq 2$ yields
\[ h(\alpha) \geq \frac1{4d} \Big(   \frac{\log \log d}{\log d}  \Big)^3.\] Let $g_1(x)= \frac{1}{4x}\big(\frac{\log \log x}{\log x}\big)^3$. The function $g_1(x)$ is positive for all $x\geq 3$ and decreasing for $x\geq 7$. For  $x=3,4,5,6$, we see that $g_1(x)$ achieves its minimum at $g_1(3)=0.00005227953369\dots$.  For $d\geq 7$, since $d\leq \tau/\epsilon$, we have
\[ h(\alpha) \geq  g_1(\tau/\epsilon).\]
There exists $a\in(184,185)$ such that for $x\leq a$ we have $f_1(x) > g_1(x)$, but for $x\geq a$, $g_1(x)>f_1(x)$. (In fact, $a=184.615\dots$.) We also note that $f_1(6)>g_1(3)$.
We conclude that when $r>\delta$, we have $h(\alpha)\geq C_1(\epsilon,\tau)$, where
\begin{align*}
C_1(\epsilon,\tau)&=
\begin{cases}
f_1(6), &\text{if } 3\leq d\leq 6 \text{ and } \tau/\epsilon<6, \\
f_1(\tau/\epsilon), &\text{if } 3\leq d\leq 6 \text{ and } \tau/\epsilon\geq 6,\text{ or if }  d\geq 7 \text{ and } \tau/\epsilon\leq a,\\
g_1(\tau/\epsilon), &\text{if } d\geq 7 \text{ and } \tau/\epsilon\geq a.\\
\end{cases}
\end{align*}
For $d=2$, we can use $C_1=f_1(2)$, and for $d=1$ we can use $C_1= \log 2$.

We may now assume that $r\leq \delta$.  Let $\rho$ be the multiplicative rank of  the conjugates of $\alpha$ over $\KK$.

First, consider the case when $\rho\geq r$. (That is, $r$ of the conjugates of $\alpha$ over $\KK$ are multiplicatively independent.)
By Proposition~\ref{AmVi1.6}, with $D=[\KK(\alpha):\QQ]=\delta \tau$, we have
\[ h(\alpha)^r \geq (\delta \tau)^{-1} (1050r^5\log(3 \delta \tau))^{-r^2(r+1)^2}\] 
using the fact that  the Weil heights of conjugate algebraic numbers are equal.
Since $r>\tfrac1{\epsilon}$, it follows that $-\tfrac1r>-\epsilon$, so $(\delta\tau)^{-\frac1r}> (\delta\tau)^{-\epsilon}$. 
Therefore, upon taking the $r^{\rm{ th}}$ roots, with $f_2(\delta, \tau,r)=(1050r^5\log(3 \delta \tau))^{-r(r+1)^2}$ we have
\begin{align}\label{eq:halphalb1} h(\alpha) > \delta^{-\epsilon}\tau^{-\epsilon}f_2(\delta, \tau,r).\end{align}
Since $r-1\leq \tfrac1{\epsilon}$, it follows that $r\leq 1+\tfrac1{\epsilon}$ so 
\[  f_2(\delta, \tau,r) \geq 
  \Big(1050(1+\tfrac1{\epsilon})^5\log(3 \delta \tau)\Big)^{-(1+\tfrac1{\epsilon})(2+\tfrac1{\epsilon})^2}.
  \]
  Using the inequality $\log(x)\leq \tfrac{1}{\epsilon_1\exp(1)} x^{\epsilon_1}$, which holds for any $\epsilon_1>0$, we see that
 \[ f_2(\delta, \tau,r) \geq \Big(1050(1+\tfrac1{\epsilon})^5 \tfrac{1}{\epsilon_{1}\exp(1)} (3\delta \tau )^{\epsilon_{1}} \Big)^{-(1+\tfrac1{\epsilon})(2+\tfrac1{\epsilon})^2}.
  \]
  Taking $\epsilon_1= \epsilon/(1+\tfrac1{\epsilon})(2+\tfrac1{\epsilon})^2$ we have 
  \[
   f_2(\delta, \tau,r) \geq  \Big( \frac{1050}{ \epsilon \exp(1)}(1+\tfrac1{\epsilon})^6  (2+\tfrac1{\epsilon})^2  \Big)^{-(1+\tfrac1{\epsilon})(2+\tfrac1{\epsilon})^2} (3 \delta \tau)^{-\epsilon}.
  \]
  We conclude, from \eqref{eq:halphalb1}, that $h(\alpha) \geq C_2(\epsilon,\tau) \delta^{-2\epsilon}$,
where 
\[ C_2(\epsilon,\tau)=\tau^{- 2\epsilon } 3^{-\epsilon}  \Big( \frac{1050}{ \epsilon \exp(1)}(1+\tfrac1{\epsilon})^6  (2+\tfrac1{\epsilon})^2  \Big)^{-(1+\tfrac1{\epsilon})(2+\tfrac1{\epsilon})^2}.\]

Now we may assume that $r\leq \delta $ and $\rho\leq r-1$. First, let us establish some notation. Let $e$ be the order of the group of roots of unity in $\KK(\alpha)$, let $f$ be the order of the group of roots of unity in $\KK$, and let $D=[\KK(\alpha):\KK(\zeta_e)]$.
By  Proposition~\ref{AmDel1.3}, taking $\AAA = \KK(\zeta_e)$ and $\BB=\KK$, we conclude that there is an absolute positive constant $c$ such that 
\[
h(\alpha) \geq   \frac{(g(\tau)\Delta)^{-c}}{D} \frac{\left(\log \log 5D\right)^3}{\left(\log 2D\right)^4},
\]
where $\Delta $ is the absolute value of the discriminant of $\KK$ over $\QQ$ and $g(\tau)=1$ if there exists a tower of successive Galois extensions 
$ \QQ=\KK_0\subset \KK_1\subset \dots \subset \KK_m=\KK$, and $g(\tau)=\tau!$ otherwise.

Notice that the function $f(x) = \frac{1}{x} \frac{\left(\log \log 5x\right)^3}{\left(\log 2x\right)^4}$ is decreasing for all $x\geq 1$.
By Lemma~\ref{lemma:orderlemma} we have 
\[ 
D=[\KK(\alpha):\KK(\zeta_e)] \leq    \delta^{\epsilon}  C_4(\KK,\epsilon)
\]
with $C_4(\KK,\epsilon) = \frac{1}{C(\epsilon)}   n(\rho) f \tau^{1+\epsilon}$.
Therefore,
\[ h(\alpha) \geq  \frac{(g(\tau)\Delta)^{-c}}{C_4(\KK,\epsilon) \delta^{\epsilon} } \frac{\big(\log \log (5C_4(\KK,\epsilon)\delta^{\epsilon})\big)^3}{\big(\log(2C_4(\KK,\epsilon)\delta^{\epsilon} )\big)^4}. \] 
It remains to show that this is $\gg_{\rho,\KK} \delta^{-2\epsilon}$.  The constant $C(\epsilon)$ from  Lemma~\ref{lem:phi} is easily seen to be positive and less than 1, which implies that $C_4(\KK,\epsilon)\geq 1$. Moreover, for all $y \geq 1$, we have
 
\[ \frac{ y\left(\log \log 5 y \right)^3}{\left(\log 2 y \right)^4} \geq \frac1{4}. \]
 Since $C_4(\KK,\epsilon)\delta^{\epsilon}\geq 1$, we conclude that
\[
 \frac{\left(\log \log ( 5 C_4(\KK,\epsilon) \delta^{\epsilon})\right)^3}{\left(\log(2 C_4(\KK,\epsilon) \delta^{\epsilon}  )\right)^4} \geq \frac1{4 C_4(\KK,\epsilon) \delta^{\epsilon}}. 
\]
We have shown that 
\[h(\alpha) \geq \delta^{-2\epsilon} C_5(\KK,\epsilon), \]
where
$C_5(\KK,\epsilon)= \frac{(g(\tau)\Delta)^{-c} C(\epsilon)^2}{4(n(\rho) f \tau^{1+\epsilon})^2}$
and $C(\epsilon)$ is the constant from Lemma~\ref{lem:phi}.
Since we are assuming that $\rho\leq r-1<1/\epsilon$, then $n(\rho)< 3^{\rho^2}< 3^{(1/\epsilon)^2}$. Thus, we have
\[h(\alpha) \geq \delta^{-2\epsilon} C_3(\KK,\epsilon), \]
where
$C_3(\KK,\epsilon)= \frac{(g(\tau)\Delta)^{-c} C(\epsilon)^2}{4(3^{(1/\epsilon)^2} f \tau^{1+\epsilon})^2}$.

\end{proof}

\begin{remark} Recall that, if $\QQ(\alpha)/\KK$ is Galois, then since $\KK\subset \QQ(\alpha)$, it follows that $\KK(\alpha)=\QQ(\alpha)$.  Therefore,  this theorem also applies to the case where $\QQ(\alpha)/\KK$ is Galois. \end{remark}

\section{Proof of Theorem~\ref{Thm2}}\label{ExplicitSection}

We prove Theorem~\ref{Thm2} below.

\begin{proof}
Let $\alpha_{1},\dots,\alpha_{\delta}$ be the conjugates of $\alpha$ over $\KK$, and let $\rho$ be their multiplicative rank. 
As a result, $\delta \leq \eta$.
Since we assume that $r$ conjugates of $\alpha$ over $\KK$ are multiplicatively independent, we know that $\rho\geq r$.\\

{\bf{Case 1 $(\rho>r)$:}}  If the multiplicative rank of the conjugates of $\alpha$ over $\KK$ is strictly larger than $r$, we know that there exists a subset
\[ 
\{\alpha_{i_{1}},\alpha_{i_{2}},\dots,\alpha_{i_{r+1}}\}\subset\{\alpha_{1},\alpha_{2},\dots,\alpha_{\delta}\}
\]
 such that $\alpha_{i_{1}},\alpha_{i_{2}},\dots,\alpha_{i_{r+1}}$ are distinct and multiplicatively independent. By Proposition~\ref{AmVi1.6},
\[
h(\alpha_{i_{1}})h(\alpha_{i_{2}})\cdots h(\alpha_{i_{r+1}})\geq D^{-1}\left(1050(r+1)^{5}\log(3D)\right)^{-(r+1)^2(r+2)^2},
\] 
where $D=[\QQ(\alpha_{i_{1}},\dots,\alpha_{i_{r+1}}):\QQ]$. Since the $\alpha_i$ are all conjugates, they all have the same height, so the left hand side of this inequality is $h(\alpha)^{r+1}$. In addition, 
\[ D=[\QQ(\alpha_{i_{1}},\dots,\alpha_{i_{r+1}}):\QQ]\leq [\FF:\QQ]= [\FF:\KK][\KK:\QQ]=\eta\tau.
\] 
Upon taking $(r+1)^{\rm{st}}$ roots, it follows that 
\[h(\alpha) \geq \tau^{-\frac{1}{r+1}} \eta^{-\frac{1}{r+1}}\left(1050(r+1)^{5}\log(3\tau\eta)\right)^{-(r+1)(r+2)^2}.\]
Recall that $\log x\leq \frac{1}{\epsilon_{1} \exp(1)}x^{\epsilon_{1}}$ for any $\epsilon_{1}>0$. By applying this inequality with $\epsilon_{1}=\frac{\epsilon}{(r+1)(r+2)^2}$, we get an explicit lower bound for $h(\alpha)$ in the desired form,
\[
h(\alpha) \geq C_1(\epsilon, r, \tau) \eta^{-\frac{1}{r+1}-\epsilon},
\]
where 
\[
C_1(\epsilon, r, \tau)=3^{-\epsilon}\left(\frac{1050(r+1)^6(r+2)^2}{\epsilon\exp(1)}\right)^{-(r+1)(r+2)^2}\tau^{-\frac{1}{r+1}-\epsilon}.
\]

\vspace{0.3cm}

{\bf{Case 2 $(\rho= r)$:}} Let $\alpha_{i_{1}},\alpha_{i_{2}},\dots,\alpha_{i_{r}}$ be multiplicatively independent conjugates of $\alpha$ over $\KK$. We denote by $e$ the order of the group of roots of unity in $\FF$ so that $\QQ(\zeta_{e})\subset \FF$. By Lemma \ref{lemma:2.1} we know that there exists a  subfield $\LL$ of $\FF$ which is Galois over $\KK$ such that $\alpha^{e}\in \LL$ and $[\LL:\KK]=n\leq n(r)< 3^{r^2}$. 
By \eqref{Vouineq}, we have 
\[ h(\alpha^{e})\geq \frac{1}{4[\QQ(\alpha^{e}):\QQ]}\left(\frac{\log\log([\QQ(\alpha^{e}):\QQ])}{\log([\QQ(\alpha^{e}):\QQ])}\right)^3,\]
provided $\alpha^e\not\in\QQ$. Now, 
\[[\QQ(\alpha^{e}):\QQ]\leq [\LL:\QQ]=[\LL:\KK][\KK:\QQ]\leq n(r)\tau.\] As in the proof of Theorem \ref{Thm1}, we can use the properties of the function $g_1(x)$ previously defined to obtain 
\begin{equation}\label{eqn:3.1}
h(\alpha)=\frac{1}{e}h(\alpha^{e})\geq\frac{1}{e}\frac{1}{4n(r)\tau}\left(\frac{\log\log(n(r)\tau)}{\log(n(r)\tau)}\right)^3
\end{equation}
whenever $[\QQ(\alpha^{e}):\QQ]\geq 7$. When $n(r)\tau\leq 184$, this can be improved by using $f_1(x)$ in place of $g_1(x)$, as before, and similarly, we use $f_1(x)$ when $3\leq [\QQ(\alpha^{e}):\QQ]\leq 6$. For $[\QQ(\alpha^{e}):\QQ]=2$, we have $h(\alpha)\geq\frac{1}{e}f_1(2)$, and for $[\QQ(\alpha^{e}):\QQ]=1$ we have $h(\alpha)\geq\frac{1}{e} \log 2$. For the remainder of the proof, we focus on the case given in equation \eqref{eqn:3.1}, and trust the reader to make the appropriate substitutions.

On the other hand, Corollary~\ref{cor:Del1.6} 
implies that with $\boldsymbol{\alpha}=(\alpha_{i_{1}},\dots,\alpha_{i_{r}})$ we have
\[
h(\alpha)^{r}> \bigg( c_{2}(r)[\QQ^{\mathrm{ab}}(\boldsymbol{\alpha}):\QQ^{\mathrm{ab}}]\Big(\log\big(3[\QQ^{\mathrm{ab}}(\boldsymbol{\alpha}):\QQ^{\mathrm{ab}}]\big)\Big)^{\kappa_{2}(r)}\bigg)^{-1},
\] 
where $\kappa_{2}(r)=3r\left(2(r+1)^2(r+1)!\right)^{r}$ and 
\[c_{2}(r)=(2r^{2})^{r}\exp\left(64r^{2}r!\left(2(r+1)^{2}(r+1)!\right)^{2r}\right).\]
Using the bound 
 \[ [\QQ^{\mathrm{ab}}(\boldsymbol{\alpha}):\QQ^{\mathrm{ab}}]\leq [\QQ(\zeta_{e})(\boldsymbol{\alpha}):\QQ(\zeta_{e})]\leq[\FF:\QQ(\zeta_{e})]=\frac{[\FF:\QQ]}{[\QQ(\zeta_{e}):\QQ]}=\frac{\tau\eta}{\phi(e)},\]
we conclude that
\begin{equation}\label{eqn:3.2}
h(\alpha)^{r}> \bigg( c_{2}(r)\frac{\tau\eta}{\phi(e)}\Big(\log\big(3\frac{\tau\eta}{\phi(e)}\big)\Big)^{\kappa_{2}(r)}\bigg)^{-1}.
\end{equation}

Combining equations \eqref{eqn:3.1} and \eqref{eqn:3.2} yields
\[
h(\alpha)^{r+1}
>
C_1(r,\tau,e) \eta^{-1}\left(\log(3\tfrac{\tau\eta}{\phi(e)})\right)^{-\kappa_{2}(r)},
\]
where $C_1(r,\tau,e) = c_{2}(r)^{-1}\big(\frac{\log\log(n(r)\tau)}{\log(n(r)\tau)}\big)^3
\frac{1}{4n(r)\tau^2}\frac{\phi(e)}{e}.$
We now apply the inequality $\log x\leq \frac{1}{\epsilon_{1} \exp(1)}x^{\epsilon_{1}}$  with $\epsilon_{1}=\epsilon/\kappa_{2}(r)$ and $x=3\tfrac{\tau\eta}{\phi(e)}$ and conclude that 
\[
h(\alpha)^{r+1}
>
C_1(r,\tau,e) \eta^{-1} \Big(  \frac{\kappa_2(r)}{\epsilon \exp(1)} \Big)^{-\kappa_2(r)} \Big(3\frac{\tau\eta}{\phi(e)}\Big)^{-\epsilon}.
\]
This simplifies to 
\[
h(\alpha)^{r+1}
>
C_2(\epsilon,r,\tau,e) \eta^{-1-\epsilon} 
\]
with $C_2(\epsilon,r,\tau,e)=\big(\frac{\log\log(n(r)\tau)}{\log(n(r)\tau)}\big)^3
\Big(  \frac{\kappa_2(r)}{\epsilon \exp(1)} \Big)^{-\kappa_2(r)}
\frac{\phi(e)^{1+\epsilon}/e}{4n(r) c_{2}(r) 3^{\epsilon} \tau^{2+\epsilon}}.$
By Lemma \ref{lem:phi}, $\phi(e)^{1+\epsilon}/e \geq C(\epsilon)$, so that 
we can replace $C_2(\epsilon,r,\tau,e)$ in the inequality by 
\[ C_3(\epsilon,r,\tau) = \Big(\frac{\log\log(n(r)\tau)}{\log(n(r)\tau)}\Big)^3
\Big(  \frac{\kappa_2(r)}{\epsilon \exp(1)} \Big)^{-\kappa_2(r)}
\frac{C(\epsilon)}{4n(r) c_{2}(r) 3^{\epsilon} \tau^{2+\epsilon}},\]
and upon taking $(r+1)^{\rm{st}}$ roots we have 
\[
h(\alpha)
>
C_3(\epsilon,r,\tau)^{\frac{1}{r+1}} \eta^{- \frac{1}{r+1}-\frac{\epsilon}{r+1}} .
\]
Using the fact that $\eta^{-\epsilon/(r+1)}\geq\eta^{-\epsilon}$ we get the bound 
\[
h(\alpha)
>
C_3(\epsilon,r,\tau)^{\frac{1}{r+1}} \eta^{- \frac{1}{r+1}-{\epsilon}} .
\]

\end{proof}

\section*{Acknowledgements}

This work began as a research project for the working group {\em Heights of Algebraic Integers} at the Women in Numbers Europe 2 workshop held at the Lorentz Center at the University of Leiden.  The authors would like to thank the organisers of the workshop, and the Lorentz Center for their hospitality.

Research of Shabnam Akhtari is supported by the NSF grant DMS-1601837. Kirsti Biggs is supported by an EPSRC Doctoral Training Partnership. Research of Alia Hamieh is partially supported by a PIMS postdoctoral fellowship. Research of Kathleen Petersen is supported by Simons Foundation Collaboration grant number 209226 and 430077; she would like to thank the Tata Institute of Fundamental Research for their hospitality while preparing this manuscript.  Lola Thompson is supported by an AMS Simons Travel Grant, by a Max Planck Institute fellowship during the Fall 2016 semester, and by the NSF grant DMS-1440140 while in residence at the Mathematical Sciences Research Institute during the Spring 2017 semester.

\end{document}